\documentclass[]{amsart}

\usepackage[english]{babel}
\usepackage[utf8]{inputenc}
\usepackage[T1]{fontenc}
\usepackage{microtype}
\usepackage{paralist}
\usepackage{amsmath, amsfonts, amssymb, amsthm}
\usepackage{color}
\usepackage[normalem]{ulem}
\usepackage{tikz}

\newcommand{\CC}{\mathbb{C}}

\newcommand{\ZZ}{\mathbb{Z}}

\newcommand{\Hc}{\mathcal{H}}

\newcommand{\Oc}{\mathcal{O}}

\newcommand{\gfr}{\mathfrak{g}}
\newcommand{\set}[1]{\left\{ #1 \right\}}
\newcommand{\setb}[1]{\left( #1 \right)}
\newcommand{\abs}[1]{\left| #1 \right|}

\newcommand{\bino}[2]{\begin{pmatrix} #1 \\ #2 \end{pmatrix}}

\newtheorem{mymasterthm}{notForUse}
\theoremstyle{definition}

\newtheorem{myrem}[mymasterthm]{Remark}

\theoremstyle{plain}
\newtheorem{mylemma}[mymasterthm]{Lemma}
\newtheorem{mythm}[mymasterthm]{Theorem}

\allowdisplaybreaks

\title[Pillai's conjecture for polynomials]{Pillai's conjecture for polynomials}
\makeatletter
\@namedef{subjclassname@2020}{\textup{2020} Mathematics Subject Classification}
\makeatother

\author[S. Heintze]{Sebastian Heintze}
\address{Sebastian Heintze\newline
	\indent Graz University of Technology\newline
	\indent Institute of Analysis and Number Theory\newline
	\indent Steyrergasse 30/II \newline
	\indent A-8010 Graz, Austria}
\email{heintze@math.tugraz.at}

\thanks{Supported by Austrian Science Fund (FWF) under project I4406}

\begin{document}
	
	\maketitle
	
	
	\begin{abstract}
		In this paper we study the polynomial version of Pillai's conjecture on the exponential Diophantine equation
		\begin{equation*}
			p^n - q^m = f.
		\end{equation*}
		We prove that for any non-constant polynomial $ f $ there are only finitely many vectors $ (n,m,\deg p,\deg q) $ with integers $ n,m \geq 2 $ and non-constant polynomials $ p,q $ such that Pillai's equation holds.
		Moreover, we will give some examples that there can still be infinitely many possibilities for the polynomials $ p,q $.
	\end{abstract}
	
	\section{Introduction}
	
	About hundred years ago Pillai \cite{pillai-1936} considered exponential Diophantine equations of the form
	\begin{equation}
		\label{eq:origpillai}
		x^n - y^m = f,
	\end{equation}
	where $ x,y,f $ are given positive integers and one is interested in integer solutions $ (n,m) $ with $ n,m \geq 2 $.
	A natural question concerning equation \eqref{eq:origpillai} is how many solutions $ (n,m) $ exist for a given triple $ (f,x,y) $.
	After some intermediate results by several authors, Bennett \cite{bennett-2001} proved that for any triple $ (f,x,y) $ of integers with $ fxy \neq 0 $ and $ x,y \geq 2 $ there are at most two solutions of equation \eqref{eq:origpillai}.
	A frequently studied generalization of Pillai's problem is to replace $ x^n $ and $ y^m $ by simple linear recurrence sequences.
	Different authors analyzed this problem for special choices of the linear recurrence sequences and Chim, Pink and Ziegler considered in \cite{chim-pink-ziegler-2018} this generalization for all simple linear recurrence sequences satisfying some natural conditions.
	Furthermore, the Pillai problem as well as its generalizations can also be considered over function fields instead of number fields.
	This was done e.g.\ by the author together with Fuchs in \cite{fuchs-heintze-p8}.
	
	So far we talked about the case that only the exponential parameters are varying and the bases are fixed.
	Already Pillai considered the situation when one allows also $ x $ and $ y $ to vary.
	He conjectured that for any given non-zero integer $ f $ equation \eqref{eq:origpillai} has at most finitely many solutions in integers $ n,m,x,y \geq 2 $.
	In the case $ f = 1 $ we get the Catalan conjecture.
	Mih\u ailescu \cite{mihailescu-2004} proved that $ 3^2 $ and $ 2^3 $ are the only perfect powers which differ exactly by $ 1 $.
	The case $ f \neq 1 $ is still an open problem.
	The purpose of the present paper is to study Pillai's conjecture for polynomials.
	More precisely, we aim for a bound on the solutions $ (n,m,p,q) $ in integers $ n,m \geq 2 $ and non-constant polynomials $ p,q \in \CC[x] $ of the Pillai equation
	\begin{equation*}
		p^n - q^m = f
	\end{equation*}
	for a given polynomial $ f \in \CC[x] $.
	Since there are still infinitely many polynomials if we bound its height (there is no version of Northcott's theorem for function fields), we may have a slightly different statement than in the integer case.
	
	\section{Results}
	
	Our theorem gives an explicit upper bound on the exponential variables and the degrees of the polynomials.
	Therefore we can deduce finiteness of the solutions only for the exponential variables and the degree of the polynomials.
	In Remark \ref{rem:counter} below we will give some counterexamples that in general finiteness does not hold for the polynomials itself.
	
	\begin{mythm}
		\label{thm:pillaiconj}
		Let $ f \in \CC[x] $ be a polynomial with $ \deg f \geq 1 $.
		Then for all solutions $ (n,m,p,q) $ in integers $ n,m \geq 2 $ and polynomials $ p,q \in \CC[x] $ with $ \deg p, \deg q \geq 1 $ of the Pillai equation
		\begin{equation}
			\label{eq:pillaieq}
			p^n - q^m = f
		\end{equation}
		we have
		\begin{equation*}
			\max \setb{n,m,\deg p,\deg q} \leq 4 + 12 \deg f + 8 (\deg f)^2.
		\end{equation*}
		In particular the vector $ (n,m,\deg p,\deg q) $ can take only finitely many different values.
	\end{mythm}
	
	Before we start with preparing the utilities for the proof of this theorem let us give some remarks on the scope of the above statement.
	
	\begin{myrem}
		We do not handle the case that $ f \in \CC $, i.e.\ a constant polynomial, in Theorem \ref{thm:pillaiconj} since this case is already solved.
		For $ f = 0 $ there are obviously infinitely many solutions for the vector $ (n,m,\deg p,\deg q) $.
		One can take e.g.\ the choice $ (n,m,p,q) = (3k,k,g,g^3) $ for any positive integer $ k \geq 2 $ and non-constant polynomial $ g \in \CC[x] $.
		The case $ f \in \CC^* $, i.e.\ constant and non-zero, is solved by Kreso and Tichy in \cite{kreso-tichy-}.
		They proved that there is no vector $ (n,m,p,q) $ with integers $ n,m \geq 2 $ and non-constant polynomials $ p,q \in \CC[x] $ which satisfies equation \eqref{eq:pillaieq} if $ f \in \CC^* $.
		Their argument essentially builds on the fact that $ f $ is constant and thus cannot be generalized to our situation.
		It is clear that there are some $ f \in \CC[x] $ for which equation \eqref{eq:pillaieq} has a solution.
	\end{myrem}
	
	\begin{myrem}
		\label{rem:counter}
		The finiteness result given in Theorem \ref{thm:pillaiconj} only holds for the vector $ (n,m,\deg p,\deg q) $.
		In general there can still be infinitely many pairs $ (p,q) $ of non-constant polynomials in $ \CC[x] $ for a given instance of the vector $ (n,m,\deg p,\deg q) $ such that equation \eqref{eq:pillaieq} holds.
		To illustrate this fact we will give some examples.
		
		Consider first the polynomial $ f = bx^k $ for a fixed integer $ k \geq 2 $ and a fixed real number $ b > 0 $. Moreover, let $ (n,m,\deg p,\deg q) = (k,k,1,1) $.
		Then we define
		\begin{align*}
			p &= \sqrt[k]{a^k+b} x \\
			q &= a x
		\end{align*}
		and get for any $ a > 0 $ the equality $ p^n - q^m = f $.
		
		Further take a look at the polynomial $ f = bx $ for a fixed real number $ b > 0 $ and let $ (n,m,\deg p,\deg q) = (2,2,1,1) $.
		Then we define
		\begin{align*}
			p &= a x + \frac{b}{4a} \\
			q &= a x - \frac{b}{4a}
		\end{align*}
		and get for any $ a > 0 $ the equality $ p^n - q^m = f $.
		
		So far we have only given examples for monomials $ f $, but there are also some for polynomials which are not monomials.
		For instance, we can generalize the previous example to more general linear polynomials.
		Consider $ f = sx+r $ for $ s \neq 0 $ and put $ (n,m,\deg p,\deg q) = (2,2,1,1) $.
		Now we define
		\begin{align*}
			p &= a x + \frac{ra}{s} + \frac{s}{4a} \\
			q &= a x + \frac{ra}{s} - \frac{s}{4a}
		\end{align*}
		and get for any $ a > 0 $ the equality $ p^n - q^m = f $.
		
		By multiplying in the last paragraph the polynomials $ p $ and $ q $ both by $ x^{\ell} $ we get the same result for the polynomial $ f = s x^{2\ell + 1} + r x^{2\ell} $.
		
		Finally, let $ f = ux^3 + tx^2 + sx $ for $ u \neq 0 $ and choose the vector $ (n,m,\deg p,\deg q) = (2,2,2,2) $.
		Here we define
		\begin{align*}
			p &= a x^2 + \frac{4a^2t + u^2}{4au} x + \frac{as}{u} \\
			q &= a x^2 + \frac{4a^2t - u^2}{4au} x + \frac{as}{u}
		\end{align*}
		and get for any $ a > 0 $ the equality $ p^n - q^m = f $.
		
		It is not clear whether or not there always, i.e.\ for any non-constant polynomial $ f \in \CC[x] $, exist infinitely many pairs $ (p,q) $ for at least one or even any instance of the vector $ (n,m,\deg p,\deg q) $, for which a solution exists at all, such that equation \eqref{eq:pillaieq} holds.
		We leave this as an open question.
	\end{myrem}
	
	\section{Preliminaries}
	
	We will work with valuations and give here for the readers convenience a short wrap-up of this notion that can e.g.\ also be found in \cite{fuchs-heintze-p8}:
	For $ c \in \CC $ and $ f(x) \in \CC(x) $, where $ \CC(x) $ is the rational function field over $ \CC $, we denote by $ \nu_c(f) $ the unique integer such that $ f(x) = (x-c)^{\nu_c(f)} p(x) / q(x) $ with $ p(x),q(x) \in \CC[x] $ such that $ p(c)q(c) \neq 0 $. Further we write $ \nu_{\infty}(f) = \deg q - \deg p $ if $ f(x) = p(x) / q(x) $.
	These functions $ \nu : \CC(x) \rightarrow \ZZ $ are up to equivalence all valuations in $ \CC(x) $.
	If $ \nu_c(f) > 0 $, then $ c $ is called a zero of $ f $, and if $ \nu_c(f) < 0 $, then $ c $ is called a pole of $ f $, where $ c \in \CC \cup \set{\infty} $.
	In $ \CC(x) $ the sum-formula
	\begin{equation*}
		\sum_{\nu} \nu(f) = 0
	\end{equation*}
	holds, where the sum is taken over all valuations (up to equivalence) in the considered function field.
	For a finite set $ S $ of valuations on $ \CC(x) $, we denote by $ \Oc_S^* $ the set of $ S $-units in $ \CC(x) $, i.e. the set
	\begin{equation*}
		\Oc_S^* = \set{f \in \CC(x)^* : \nu(f) = 0 \text{ for all } \nu \notin S}.
	\end{equation*}
	
	The proof of Theorem \ref{thm:pillaiconj} will use height functions. Hence, let us define the height of an element $ f \in \CC(x)^* $ by
	\begin{equation*}
		\Hc(f) := - \sum_{\nu} \min \setb{0, \nu(f)} = \sum_{\nu} \max \setb{0, \nu(f)}
	\end{equation*}
	where the sum is taken over all valuations (up to equivalence) on the rational function field $ \CC(x) $. Additionally we define $ \Hc(0) = \infty $.
	This height function satisfies some basic properties, listed in the lemma below which is proven in \cite{fuchs-karolus-kreso-2019}:
	
	\begin{mylemma}
		\label{lemma:heightproperties}
		Denote as above by $ \Hc $ the height on $ \CC(x) $. Then for $ f,g \in \CC(x)^* $ the following properties hold:
		\begin{enumerate}[a)]
			\item $ \Hc(f) \geq 0 $ and $ \Hc(f) = \Hc(1/f) $,
			\item $ \Hc(f) - \Hc(g) \leq \Hc(f+g) \leq \Hc(f) + \Hc(g) $,
			\item $ \Hc(f) - \Hc(g) \leq \Hc(fg) \leq \Hc(f) + \Hc(g) $,
			\item $ \Hc(f^n) = \abs{n} \cdot \Hc(f) $,
			\item $ \Hc(f) = 0 \iff f \in \CC^* $,
			\item $ \Hc(A(f)) = \deg A \cdot \Hc(f) $ for any $ A \in \CC[T] \setminus \set{0} $.
		\end{enumerate}
	\end{mylemma}
	
	Moreover, the following theorem due to Brownawell and Masser is an important ingredient for our proof. It is an immediate consequence of Theorem B in \cite{brownawell-masser-1986}:
	
	\begin{mythm}[Brownawell-Masser]
		\label{thm:brownawellmasser}
		Let $ F/\CC $ be a function field in one variable of genus $ \gfr $. Moreover, for a finite set $ S $ of valuations, let $ u_1,\ldots,u_k $ be $ S $-units and
		\begin{equation*}
			1 + u_1 + \cdots + u_k = 0,
		\end{equation*}
		where no proper subsum of the left hand side vanishes. Then we have
		\begin{equation*}
			\max_{i=1,\ldots,k} \Hc(u_i) \leq \bino{k}{2} \left( \abs{S} + \max \setb{0, 2\gfr-2} \right).
		\end{equation*}
	\end{mythm}
	
	\section{Proof}
	
	We have now all tools that we need for proving our theorem. Hence we can start with the proof:
	
	\begin{proof}[Proof of Theorem \ref{thm:pillaiconj}]
		Since $ f $ is not zero, we can rewrite equation \eqref{eq:pillaieq}, by dividing by $ f $ and bringing all terms to one side, as
		\begin{equation*}
			1 + \frac{q^m}{f} - \frac{p^n}{f} = 0.
		\end{equation*}
		The left hand side of this equation cannot have a vanishing subsum since there are only three summands and none of them is zero.
		Let $ S $ be a finite set of valuations such that $ p,q,f $ are all $ S $-units.
		This is possible with
		\begin{equation*}
			\abs{S} \leq 1 + \deg f + \deg p + \deg q.
		\end{equation*}
		Now we can apply Theorem \ref{thm:brownawellmasser} and get
		\begin{equation*}
			\Hc \left( \frac{p^n}{f} \right) \leq \abs{S} \leq 1 + \deg f + \deg p + \deg q
		\end{equation*}
		as well as
		\begin{equation*}
			\Hc \left( \frac{q^m}{f} \right) \leq \abs{S} \leq 1 + \deg f + \deg p + \deg q
		\end{equation*}
		because the genus of $ \CC(x) $ is $ 0 $.
		Using Lemma \ref{lemma:heightproperties} and the definition of the height this yields
		\begin{equation}
			\label{eq:ndegp}
			n \cdot \deg p = \Hc(p^n) \leq 1 + 2 \deg f + \deg p + \deg q
		\end{equation}
		and
		\begin{equation}
			\label{eq:mdegq}
			m \cdot \deg q = \Hc(q^m) \leq 1 + 2 \deg f + \deg p + \deg q.
		\end{equation}
		
		In the sequel we may assume without loss of generality that $ \deg p \geq \deg q $.
		The other case is completely analogous.
		For a better readability we give a name to the bound stated in the theorem, say
		\begin{equation*}
			B := 4 + 12 \deg f + 8 (\deg f)^2.
		\end{equation*}
		Now inequality \eqref{eq:ndegp} implies
		\begin{equation}
			\label{eq:makecases}
			(n-2) \cdot \deg p \leq 1 + 2 \deg f.
		\end{equation}
		We have to distinguish between three cases.
		The first case is $ n \geq 3 $, the second case supposes $ n = 2 $ and $ m \geq 3 $, and in the last case we consider the situation $ n = m = 2 $.
		
		So let us start with the case $ n \geq 3 $.
		Then inequality \eqref{eq:makecases} immediately gives
		\begin{equation*}
			\deg p \leq 1 + 2 \deg f \leq B
		\end{equation*}
		as well as
		\begin{equation*}
			n \leq 3 + 2 \deg f \leq B.
		\end{equation*}
		By the original equation \eqref{eq:pillaieq} we get
		\begin{align*}
			m \cdot \deg q &= \deg (q^m) = \deg (p^n-f) \leq \max \setb{n \cdot \deg p, \deg f} \\
			&\leq (1 + 2 \deg f)(3 + 2 \deg f) = 3 + 8 \deg f + 4 (\deg f)^2
		\end{align*}
		which yields
		\begin{equation*}
			\deg q \leq 3 + 8 \deg f + 4 (\deg f)^2 \leq B
		\end{equation*}
		and
		\begin{equation*}
			m \leq 3 + 8 \deg f + 4 (\deg f)^2 \leq B.
		\end{equation*}
		
		As the next step we consider the case $ n = 2 $ and $ m \geq 3 $.
		Here inequality \eqref{eq:ndegp} is equivalent to
		\begin{equation*}
			\deg p \leq 1 + 2 \deg f + \deg q.
		\end{equation*}
		Inserting this into inequality \eqref{eq:mdegq} gives
		\begin{equation*}
			m \cdot \deg q = \Hc(q^m) \leq 2 + 4 \deg f + 2 \deg q
		\end{equation*}
		and thus
		\begin{equation*}
			(m-2) \cdot \deg q \leq 2 + 4 \deg f.
		\end{equation*}
		Therefore we have the upper bounds
		\begin{equation*}
			\deg q \leq 2 + 4 \deg f \leq B
		\end{equation*}
		and
		\begin{equation*}
			m \leq 4 + 4 \deg f \leq B.
		\end{equation*}
		Using once again equation \eqref{eq:pillaieq} we also get
		\begin{align*}
			2 \cdot \deg p &= \deg (p^2) = \deg (q^m+f) \leq \max \setb{m \cdot \deg q, \deg f} \\
			&\leq (2 + 4 \deg f)(4 + 4 \deg f) = 8 + 24 \deg f + 16 (\deg f)^2
		\end{align*}
		which yields
		\begin{equation*}
			\deg p \leq 4 + 12 \deg f + 8 (\deg f)^2 = B.
		\end{equation*}
		
		Finally we consider the case $ n = m = 2 $.
		Here we can factorize equation \eqref{eq:pillaieq} to get
		\begin{equation*}
			(p-q) (p+q) = f.
		\end{equation*}
		Since $ f $ is non-zero, we have $ p-q \neq 0 $ and $ p+q \neq 0 $ and the two bounds
		\begin{align*}
			\deg (p-q) &\leq \deg f, \\
			\deg (p+q) &\leq \deg f.
		\end{align*}
		If $ \deg p > \deg q $, then we have
		\begin{equation*}
			\deg q < \deg p = \deg (p-q) \leq \deg f \leq B.
		\end{equation*}
		Otherwise, when $ \deg p = \deg q $, we get
		\begin{equation*}
			\deg p = \deg q = \deg (p+q) \leq \deg f \leq B
		\end{equation*}
		if $ p $ and $ q $ have the same leading coefficient, or
		\begin{equation*}
			\deg p = \deg q = \deg (p-q) \leq \deg f \leq B
		\end{equation*}
		if $ p $ and $ q $ have different leading coefficients, respectively.
		Thus the theorem is proven.
	\end{proof}
	
	\begin{myrem}
		Most of the parts of the proof can be generalized to the more general equation
		\begin{equation*}
			a p^n + b q^m = f
		\end{equation*}
		for a given non-constant polynomial $ f $ and given non-zero polynomials $ a,b $ in $ \CC[x] $.
		Again we are interested in solutions $ (n,m,p,q) $ in integers $ n,m \geq 2 $ and polynomials $ p,q \in \CC[x] $ with $ \deg p, \deg q \geq 1 $.
		Here the result is almost analogous to them of Theorem \ref{thm:pillaiconj} above, where now the upper bound for $ \max \setb{n,m,\deg p,\deg q} $ clearly also depends on $ \deg a $ and $ \deg b $.
		The essential difference in the general situation is that, if $ a $ and $ b $ are not both constant, we have to exclude the case $ n = m = 2 $ in the assumptions of the theorem in order to achieve that our proof still works since our factorization argument in the proof does not apply to this situation.
		We leave it up to the interested reader to calculate the upper bound in the generalized setting explicitly.
		
		Note that the exclusion of $ n = m = 2 $ has a parallel to the (still open) problem for integers.
		In the generalized equation
		\begin{equation*}
			a x^n + b y^m = f
		\end{equation*}
		with given non-zero integers $ a,b,f $ to be solved in positive integers $ n,m,x,y $, for the seeked finiteness result one obviously has to exclude the case $ n = m = 2 $ as there arise (for suitable $ a,b $ and $ f $) infinitely many solutions $ (x,y) $ from the related Pell equation if the unit rank of the involved number field is larger than zero.
	\end{myrem}

\end{document}